\documentclass[12 pt]{amsart}
\usepackage{amscd,amssymb,amsthm}
\usepackage[arrow,matrix]{xy}
\usepackage{graphicx}
\oddsidemargin=0.3in \evensidemargin=0.3in \theoremstyle{plain}
\newtheorem{thm}[subsection]{Theorem}
\newtheorem{lem}[subsection]{Lemma}

\newtheorem{cor}[subsection]{Corollary}

\theoremstyle{definition}
\newtheorem{rk}[subsection]{\textmd{Remark}}

\numberwithin{equation}{section} 

\begin{document}
\title [Solution of Certain Pell Equations]{Solution of Certain Pell Equations}
\author{Zahid Raza, Hafsa Masood Malik}
 \address{Department of Mathematics, \\ National University of Computer and Emerging Sciences, B-Block, Faisal Town, Lahore,
         Pakistan.}
\email{zahid.raza@nu.edu.pk, hafsa.masood.malik@gmail.com}



\begin{abstract}
Let $a,b,c $ be any positive integers such that $c\mid ab$ and $d_i^\pm$ is a  square free positive integer of the form $d_i^\pm=a^{2k} b^{2l}\pm i c^m$ where $k,l \geq m$ and $i=1,2.$ The main focus of this paper to find the fundamental solution of the equation $ x^2-d_i^\pm y^2=±1,$ with the help of the continued fraction of $\sqrt{d_i^\pm}.$
We also obtain all the positive solutions of the equations $ x^2-d_i^\pm y^2=\pm 1$ and $ x^2-d_i^\pm y^2=\pm 4$ by means of the Fibonacci and Lucas sequences.

Furthermore, in this work, we derive some algebraic relations on the Pell form $ F_{d_i^\pm}(x, y) = x^2-d_i^\pm y^2 $ including cycle, proper cycle, reduction and proper automorphism of it. We also determine the integer solutions of the Pell equation  $ F_{\Delta _{d_i^\pm}} (x, y) = 1 $ in terms of  $d_i^\pm.$

We generalized  all the results of the papers  \cite{Bw}, \cite{G},  \cite{P} and \cite{Atz}.
\end{abstract}

\maketitle

\section{Introduction}
 Let $ d $  be a positive integer which is not a perfect square and $ N $ be any nonzero fixed integer. Then the equation $ x^2 - d y^2 = N $ is known as Pell equation after the name of English mathematician, John Pell. The equations $ x^2 - d y^2 = 1$ and $ x^2 - d y^2 = -1 $ are known as the classical Pell equations. If $ a^2 - d b^2 = N,$ we say that $(a, b)$ is a solution of the Pell equation $ x^2 - d y^2 = N.$ We use the notation $ (a, b)$ and $ a + \sqrt{d} b $ interchangeably to denote the solutions of the equation $ x^2 - d y^2 = N.$ Also, if $ a $ and $ b $ are positive, we say that $ a + b \sqrt{d}$ is a positive solution to the equation $ x^2-dy^2=N$. Among these there is a least solution $ a_{1} + b_{1} \sqrt{d} $, in which $ a_{1} $ and $ b_{1} $ have their least positive values. Then the number $ a_{1} + b_{1} \sqrt{d}  $ is called fundamental solution of the equation $  x^2 -d y^2 = N.$  If $ a + \sqrt{d} b $ and $ r + \sqrt{d} s $  are solutions of the equation $ x^2 - d y^2 = N $, then $ a=r $  iff $  b = s $, and $ a+\sqrt{d} b< r + \sqrt{d} s $ iff $ a < r $ and $ b < s $.

 An equation $x^2-dy^2=1$ has infinite many solutions iff the equation $x^2-dy^2=-1$ has no solution. The continued fraction of $\sqrt{d}$ played a vital role to solve the Pell equation $ x^2-dy^2=\pm 1$. Actually its period length is useful for knowing the solution of this equation.
 Let $d$ be a positive integer that is not a perfect square. Then there is a continued fraction expansion of $ \sqrt{d} $ such that
 $ \sqrt{d}=[a_{0},\overline{a_{1} , a_{2} ,\ldots, a_{n-1} ,2a_{0} }]$ where $ n $ is the period length and the ${a_{j}}'s $ are given by the recursion formula;
    $a_{0}=\sqrt{d} , a_{k} = \llcorner a_{k} \lrcorner  $
    and
     $ a_{k+1} = \frac{1}{x_{k}-a_{k}},k=0,1,2,\ldots$

    Recall that $ a_{n} = 2 a_{0}  $ and $ a_{n+k}=a_{k}$ for all $k \geq 1.$ Then $ n^{th}  $ convergent of$ \sqrt{d} $   is given by
    $$\frac{p_{n}}{q_{n}} =[a_{0},\overline{ a_{1} , a_{2} , \ldots , a_{n-1}, a_{n}}  ] = a_{0} + \frac{1}{a_{1}\frac{1}{a_{2} + \frac{1}{\ddots \frac{1}{a_{n-1}+\frac{1}{a_{n}}}}}}  \ \ \ \ n \geq 0$$

In this paper, we give the fundamental solution of the equation $ x^2 - d_i^\pm y^2 = \pm 1 $ by means of the period length of the continued fraction expansion of $ \sqrt{d_i^\pm}$, where $ d_i^\pm={a^{2k} b^{2l}\pm ic^m }.$
After finding the fundamental solution of the Pell equation  $x^2-d_i^\pm y^2=\pm 1,$  we obtain the positive integer solutions of equation  $ x^2-d_i^\pm y^2=\pm 4,$ for  $d_i^\pm=a^{2k} b^{2l}\pm ic^m$ by the means of the generalized Fibonacci and Lucas sequences. The main results of this paper also generalized the results presented in \cite{Bw}, \cite{G} and \cite{P}.

Furthermore, in this work, we derive some algebraic relations on the Pell form $ F_{d_i^\pm}(x, y) =
x^2-d_i^\pm y^2 $ or of discriminant $ \Delta _{d_i^\pm} = 4d_i^\pm $ including cycle, proper cycle, reduction
and proper automorphism of it. Also we determine the integer solutions of
the Pell equation  $ F_{\Delta _{d_i^\pm}} (x, y) = 1 $ via $d_i^\pm$. The main results of this paper also generalized the results presented in \cite{Atz}.

\section{Basic Setup}

 If $ N$ is a quadratic non-residue modulo $d,$ then the Pell Equation $x^2 - dy^2 = N$
 has no integer solution. If $N$ is a perfect square, then the Pell Equation $x^2 - dy^2 = N$ is solvable in
integers for all positive, non-square integers $d$.
The equation $x^2 - dy^2 = 1$ has a solution in positive integers $x$ and $y$ for all
positive, non-square integers $d$.
If $(x, y) = (u, v)$ is a positive integer solution of $x^2-dy^2 = 1$, then there exists
a positive integer $m$ such that $u + v\sqrt{d}= x_1 + y_1 m\sqrt{d}$
, where $(x_1, y_1)$ is the fundamental
solution of $x^2 - dy^2 = 1$.

If we know the fundamental solutions of the equations $ x^2 - d y^2 = \pm 1$ and $ x^2 - d y^2 = \pm 4 $ then we can give all positive integer solutions to these equations. For more information about the Pell equation, one can consult \cite{N} and \cite{L}.

The generalized Fibonacci and Lucas sequences $f_{n} (w,z)$ and $ L_{n}(w,z)$ are given in the followings:

 Let $w$ and $z$ be two nonzero positive integers with $ w^2+4z\geq 0$. The generalized Fibonacci and Lucas sequences with initial conditions $ f_{0} (w,z)=0,f_{1} (w,z)=1 $ and $ L_{0} (w,z)=2,L_{1} (w,z)=0$ are of the form
$ f_{n} (w,z)=w f_{n-1} (w,z)+zf_{n-2} (w,z)$,
$ L_{n} (w,z)=w L_{n-1} (w,z)+z L_{n-2} (w,z)$
$\forall \ n \geq 2$,
respectively.
They can also be represented by the closed formula
$f_{n} (w,z)=\frac{\alpha ^n-\beta ^n}{\alpha -\beta}$
and
   			$L_{n} (w,z)=\alpha^n + \beta^n, $
where  $\alpha =\frac{w + \sqrt{w^2+4z}}{2},$ $ \beta=\frac{w-\sqrt{w^2+4z}}{2}.$ This  identity is well known as Binet's formula.
It is easy to see that $ \alpha + \beta =w, \alpha -\beta = \sqrt{w^2+4z} $ and $ \alpha\beta=-z. $
For more information about the generalized Fibonacci and Lucas sequences, one can consult  \cite{Kl}, \cite{Rbbi}, \cite{Mc} and \cite{Mh}.
   \begin{lem}\label{l1}
   Let $\frac{p_k}{q_k}$ be the convergent of the continued fraction expansion of $\sqrt{d},$ and let $l$ be the length of the expansion.
   \begin{itemize}
     \item If $l$ is even, then the fundamental solution of $ x^2 - d y^2 = 1 $ is given by $$x=p_{l-1} \ \ \ y=q_{l-1} $$ and the equation $ x^2 - d y^2 = -1 $ has no solutions.
     \item If $l$ is odd, then the fundamental solution of $ x^2 - d y^2 = 1 $ is given by $$x=p_{2l-1} \ \ \ y=q_{2l-1} $$ and $x=p_{l-1}, y=q_{l-1} $ is the fundamental solution of $ x^2 - d y^2 = -1$.
   \end{itemize}
       \end{lem}

    \begin{thm}\label{t1}
      If $ x_{1}, y_{1}$ is the fundamental solution of $ x^2 - d y^2 = 1,$ then every positive solution of the equation is given by $ x_{n},$ $y_{n}$, where $x_n$ and $y_n$ are integers determined from $x_n+y_n\sqrt{d}=(x_{1} + y_{1}\sqrt{d})^n $ \ \ \ \ $n=1,2,3,\ldots$
    \end{thm}
    \begin{thm}\label{t2}
      If $ x_{1}, y_{1}$ is the fundamental solution of $ x^2 - d y^2 =-1,$ then every positive solution of the equation is given by $ x_{n},$ $y_{n}$, where $x_n$ and $y_n$ are integers determined from $x_n+y_n\sqrt{d}=(x_{1} + y_{1}\sqrt{d})^{2n-1} $ \ \ \ \ $n=1,2,3,\ldots$
      \end{thm}
The following two theorems are given in \cite{L}.

     \begin{thm}\label{t3}
      If $ x_{1}, y_{1}$ is the fundamental solution of $ x^2 - d y^2 =4,$ then every positive solution of the equation is given by $ x_{n},$ $y_{n}$, where $x_n$ and $y_n$ are integers determined from $ x_{n} + y_{n} \sqrt{d} = \frac{(x_{1}+y_{1} \sqrt{d} )^n }{2^{n-1}}$ \ \ \ \ $n=1,2,3,\ldots$
 \end{thm}

  \begin{thm}\label{t4}

   If $ x_{1}, y_{1}$ is the fundamental solution of $ x^2 - d y^2 =-4,$ then every positive solution of the equation is given by $ x_{n},$ $y_{n}$, where $x_n$ and $y_n$ are integers determined from $ x_{n} + y_{n} \sqrt{d} = \frac{(x_{1}+y_{1} \sqrt{d} )^n }{4^{n-1}}$ \ \ \ \ $n=1,2,3,\ldots$
   \end{thm}
The following theorems are given in \cite{Rbts}.

\begin{thm}\label{t8}
Let $d \equiv 2 \ (mod4)$  or $d\equiv3 \ (mod4).$ Then the equation $ x^2-d y^2= - 4 $ has no solution if and only if the equation $ x^2- dy^2=-1 $  has positive solutions.
\end{thm}
\begin{thm}\label{t9}
Let $d\equiv0 \ (mod4).$ If $x_{1}+\frac{d}{4}y_{1}$ is the fundamental solution  of the equation  $ x^2-\frac{d}{4}y^2=1 $, then the fundamental solution of the equation $ x^2-dy^2=4  $ is given as $ (2x_{1},y_{1}).$
\end{thm}
\begin{thm}\label{t10}
Let $ d\not\equiv0 \ (mod4).$ If $x_{1}+y_{1} \sqrt{d} $ is the fundamental solution of the equation $ x^2-dy^2=1 $ then the fundamental solution of the equation $ x^2-dy^2=4 $ is $ (2x_{1} 2y_{1}).$
\end{thm}
\section{Basic Setup 2}
A real binary quadratic form (or just a form) $F$ is a
polynomial in two variables $ x$ and $ y $ of the type
\begin{equation}\label{e1}
  F = F(x, y) = ax^2 + bxy + cy^2
\end{equation}
 with real coefficients $ a, b, c.$ We denote $F$ briefly by $ F = (a, b, c).$ The discriminant of $F$ is defined by the formula $ b^2 - 4ac $ and is denoted by $ \Delta.$ A quadratic form $F$ of discriminant $ \Delta $  is called indefinite if $ \Delta > 0,$ and is called integral if
and only if $ a, b, c \in \mathbb{Z}.$  An indefinite quadratic form $ F = (a, b, c) $ of discriminant $ \Delta $  is said to be reduced if
\begin{equation}\label{e2}
   |\sqrt{\Delta}-2|a||<b<\sqrt{ \Delta}
\end{equation}

Most properties of quadratic forms can be giving by the aid of extended
modular group $ \overline{\Gamma} $ (see \cite{Ma}). Gauss defined the group action of $ \overline{\Gamma} $ on the set of
forms as follows:
\begin{equation}\label{e3}
gF(x, y) =\ \ (ar^2 + brs + cs^2)x^2 + (2art  + bru + bts + 2csu) xy+(at^2 + btu + cu^2)y^2
\end{equation}
for $ g = \left(
          \begin{array}{cc}
            r  & s \\
            t & u \\
          \end{array}
        \right)\in \overline{\Gamma}.$
 An element $ g \in \overline{\Gamma} $ is called an automorphism of $F$ if $gF = F.$
If $\det g = 1,$ then $g$ is called a proper automorphism of $F$ and if $\det g = - 1,$ then
$g$ is called an improper automorphism of $F$. Let $ Aut(F)^+ $ denote the set of proper
automorphisms of $F$ and let $Aut(F)^- $ denote the set of improper automorphisms
of $F$ (for further details on binary quadratic forms see \cite{J1}, \cite{D}, \cite{De} and \cite{R1}).
Let $ \rho(F)$ denotes the normalization (it means that replacing $F$ by its normalization) of $(c,-b, a).$ To be more explicit, we set
\begin{equation}\label{e4}
    \rho ^{j+1}(F) = (c_{j},-b_{j} + 2c_{j}r_{j}, c_{j}r^2_{j} - b_{j}r_{j} + a_{j}),
\end{equation}
where
\begin{equation}\label{e5}
r_{j}=\left\{
  \begin{array}{ll}
   \mbox{sign}(c_{j}) \Big\lfloor \frac{b_{j}}{2|c_{j}|}\Big\rfloor, & \hbox{for $|c_{j}|\geq\sqrt{\Delta}$;} \\\\
    \mbox{sign}(c_{j}) \Big\lfloor \frac{b_{j}+\sqrt{\Delta}}{2|c_{j}|}\Big\rfloor, & \hbox{for $|c_{j}|<\sqrt{\Delta}$.}
  \end{array}
\right.
  \end{equation}

for $ j \geq 0.$ The number $r$ j's called the reducing number and the form $\rho ^{j+1}(F) $
is called the reduction of $F.$ Further if $F$ is reduced, then so is $\rho ^{j+1}(F) $. In
fact, $\rho$ is a permutation of the set of all reduced indefinite forms. Let $ \tau(F) =
\tau(a, b, c) = (-a, b,-c).$  Then the cycle of $F$ is the sequence $((\tau\rho)^j(G))$ for $ j\in  \mathbb{Z},$
where $G = (A,B,C)$ is a reduced form with $ A > 0 $ which is equivalent to $F.$
The cycle and proper cycle of $F$ is given by the following theorem \cite{J}:
\begin{thm}\label{s1}
 Let $F = (a, b, c)$ be reduced indefinite quadratic form of
discriminant $ \Delta.$ Then the cycle of $F$ is a sequence $ F_{0} \sim  F_{1} \sim  F_{2}\sim \ldots \sim  F_{l-1} $ of
length $l,$ where $ F_{0} = F = (a_{0}, b_{0}, c_{0}),$
\begin{equation}\label{e6}
  s_{j}= s(F_{j})=\Big\lfloor\frac{ b_{j}+\sqrt{\Delta}}{2|c_{j}|}\Big\rfloor \ \ \ \mbox {and}
\end{equation}
\begin{equation}\label{e7}
  F_{j+1}=(a_{j+1}, b_{j+1}, c_{j+1})=(|c_{j+}|,-b_{j} + 2|c_{j}|s_{j}, -(c_{j}s^2_{j} + b_{j}s_{j} + a_{j}))
\end{equation}
for $ 1 \leq i \leq l - 2.$ If $l$ is odd, then the proper cycle of $F$ is
$ F_{0} \sim  \tau F_{1} \sim  F_{2}\sim \tau F_{3} \cdots \sim \tau F_{l-2}\sim  F_{l-1} \sim  \tau F_{0} \sim   F_{1} \sim  \tau F_{2}\sim F_{3} \cdots \sim  \sim F_{l-2}\tau F_{l-1} $ of length $2l$. In this case the equivalence class of $F$ is equal to the proper equivalence class of $F$, and if $l$ is even, then the proper cycle of $F$ is
$ F_{0} \sim  \tau F_{1} \sim  F_{2}\sim \tau F_{3} \ldots \sim  F_{l-2}\sim  \tau F_{l-1}$
of length $l.$ In this case the equivalence class of $F$ is the disjoint union of the
proper equivalence class of $F$ and the proper equivalence class of $\tau(F)$.
\end{thm}

  \section{Main Results}
  In this section, we will give our main results about the positive solutions of the Pell equations $x^2-d y^2=\pm 1,$ and  $ x^2-dy^2=\pm 4.$ for particular values of $d$. More precisely, for $d=d_i^\pm=a^{2k} b^{2l}\pm ic^m,$  all the positive solutions of the equation in terms of the generalized Fibonacci and Lucas sequences has been investigated. Throughout in this section $h$ will be a positive integer such that $h=\frac{a^kb^l}{c^m}$,
  because $c\mid ab$.
  \begin{thm}\label{t5}
   Let $h=\frac{a^kb^l}{c^m}$ be a positive integer, then  the continued fraction of \\
   \begin{description}
                  \item[i] $d_1^-$ has of the form $[a^k b^l  ; \overline{1,2h-2,1,2a^kb^l-2}]\ \ \ \ ab\geq2$
                   \item[ii]  $d_1^+$ has of the form  $[a^k b^l  ; \overline{2h,2a^k b^l}]$
                  \item[iii] $d_2^-$ has of the form $[a^k b^l-1  ; \overline{1,h-2,1,2a^kb^l-2}]\ \ \ \ ab\geq3$
                   \item[iv]  $d_2^+$ has the form  $[a^k b^l  ; \overline{h,2a^k b^l}]$
                      \end{description}
     \end{thm}

 \begin{proof}
 For $d_1^-,$ the continued fraction
 \begin{eqnarray*}
    \sqrt{a^{2k} b^{2l} - c^m } &=& a^kb^l-1 + \sqrt{a^{2k}b^{2l} -c^m}-(a^kb^l-1) \\
                &=& a^kb^l-1 + \frac{ (\sqrt{a^{2k}b^{2l} -c^m}-(a^kb^l-1)) (\sqrt{a^{2k}b^{2l} -c^m}+a^kb^l-1)}{\sqrt{a^{2k}b^{2l} -c^m}+a^kb^l-1}\\
    &=& a^kb^l-1 +  \frac{1}{\frac{\sqrt{a^{2k}b^{2l} -c^m}+a^kb^l-1}{2a^kb^l-c^m-1}}\\
                &=& a^kb^l-1 +  \frac{1}{1+\frac{ \sqrt{a^{2k}b^{2l} -c^m} - (a^kb^l-c^m)}{2a^kb^l-c^m-1}}\\
                &=& a^kb^l-1 +  \frac{1}{1+\frac{ c^m}{\sqrt{a^{2k}b^{2l} -c^m} + a^kb^l-c^m}}\\
                &=& a^kb^l-1 +  \frac{1}{1+\frac{1}{\frac{\sqrt{a^{2k}b^{2l} -c^m} + a^kb^l-c^m}{c^m}}}\\
                &=& a^kb^l-1 +  \frac{1}{1+\frac{1}{2h-2+\frac{\sqrt{a^{2k}b^{2l} -c^m} -( a^kb^l-c^m)}{2a^kb^l-c^m-1}}}\\
                &=& a^kb^l-1 +  \frac{1}{1+\frac{1}{2h-2+\frac{1}{\frac{\sqrt{a^{2k}b^{2l} -c^m} + (a^kb^l-c^m)}{2a^kb^l-c^m-1}}}}\\
                &=& a^kb^l-1 +  \frac{1}{1+\frac{1}{2h-2+\frac{1}{1+\frac{\sqrt{a^{2k}b^{2l} -c^m} - (a^kb^l-1)}{2a^kb^l-c^m-1}}}}\\
                &=& a^kb^l-1 +  \frac{1}{1+\frac{1}{2h-2+\frac{1}{1+\frac{1}{\sqrt{a^{2k}b^{2l} -c^m} + (a^kb^l-1)}}}}\\
                &=& a^kb^l-1 +  \frac{1}{1+\frac{1}{2h-2+\frac{1}{1+\frac{1}{2a^kb^l-2+\sqrt{a^{2k}b^{2l} -c^m} - (a^kb^l-1)}}}} \\
 \end{eqnarray*}
   Hence $\sqrt{a^{2k} b^{2l}-c^m}$  has the continued fraction of the form $[a^k b^l-1 ; \overline{1,2h-2,1,2a^k b^l-2}].$
  \\Similarly for $d_2^-$, one can obtained the required form of the continued fraction.    \\
   For $d_2^+,$ the continued fraction
   \begin{eqnarray*}
     \sqrt{a^{2k} b^{2l} + 2 c^m } &=& a^k b^l + (\sqrt{a^{2k} b^{2l} + 2 c^m }-a^k b^l )\\
      &=& a^k b^l + \frac{(\sqrt{a^{2k} b^{2l} + 2 c^m} -a^k b^l )(\sqrt{a^{2k} b^{2l} + 2c^m }+a^k b^l)}{\sqrt{a^{2k} b^{2l} + 2c^m } + a^k b^l}\\
      &=& a^k b^l+\frac{1}{\frac{\sqrt{a^{2k} b^{2l} + 2c^m}+a^k b^l} {2c^m}}= a^k b^l+\frac{1}{h+\frac{\sqrt{a^{2k} b^{2l}+ 2c^m}-a^k b^l}{2c^m}}\\
      &=& a^k b^l+\frac{1}{h+\frac{1}{\sqrt{a^{2k} b^{2l}+ 2c^m}-a^kb^l}}=   a^k b^l\frac{1}{h+\frac{1}{2a^k b^l +\sqrt{a^{2k} b^{2l}+2c^m }-a^k b^l}}
   \end{eqnarray*}
    Hence $\sqrt{a^{2k} b^{2l}+2c^m}$  has the continued fraction of the form $[a^k b^l ; \overline{h,2a^k b^l}].$
    \\Similarly for $d_1^+$, one can obtained the required form of the continued fraction.
  \end{proof}

\begin{cor}
  If $c=1$,  then the continued fraction of
   \begin{description}
     \item[i] $d_1^-$ has of the form $[a^k b^l  ; \overline{1,2a^kb^l-2}]$
                   \item[ii]  $d_1^+$ has the form  $[a^k b^l  ; \overline{2a^k b^l}]$
\item[iii] $d_2^-$ has of the form $[a^k b^l -1 ; \overline{1,a^kb^l-2,1,2a^kb^l-2}]$
                   \item[iv]  $d_1^+$ has the form  $[a^k b^l  ; \overline{a^kb^l,2a^k b^l}]$
                   \end{description}
     \end{cor}
\begin{rk}
The continued fraction of $d_3=\sqrt{a^{2k}-a^k}$ is of the form  $[a^k-1  ; \overline{2,2a^k -2}]$
\end{rk}
 \begin{thm}\label{t6}
 \begin{description}
                   \item[i] Let us consider the Pell equation $ x^2-d_1^\pm y^2=1,$ then the fundamental solution $ (x_{1}^{\pm 1},y_{1}^{\pm 1} )$ is of the form $(2h a^k b^l \pm 1 ,2h)$ and the other solutions are $ (x_{n}^{\pm1}, y_{n}^{\pm1}),$ where $$\frac{x_{n}^{+1}}{y_{n}^{+1}}=[a^kb^l;{\underbrace{2h,2a^kb^l}_{ (n-1) time}},2h]\ \ \ \ \ \ \ \ \ \ \mbox{and }$$ $$\frac{x_{n}^{-1}}{y_{n}^{-1}}=[a^kb^l-1;{\underbrace{1,2h-2,1,2a^kb^l-2}_{(n-1) time}},1]\ \ \ \ \ \ \ \ \ \ \mbox{ respectively.}$$
                  \item[ii]   Let us consider the Pell equation $ x^2-d_2^\pm y^2=1,$ the fundamental solution $ (x_{1}^{\pm2},y_{1}^{\pm2} )$ is of the form $(h a^k b^l \pm1 ,h)$ and the other solutions are $ (x_{n}^{\pm2}, y_{n}^{\pm2}),$ where $$\frac{x_{n}^{+2}}{y_{n}^{+2}}=[a^kb^l;{\underbrace{h,2a^kb^l}_{ (n-1) time}},h] \ \ \ \ \ \ \ \ \ \ \mbox{and }$$ $$\frac{x_{n}^{-2}}{y_{n}^{-2}}=[a^kb^l-1;{\underbrace{1,h-2,1,2a^kb^l-2}_{(n-1) time}},1]\ \ \ \ \ \ \ \ \ \ \mbox{ respectively. }$$
                 \end{description}
\end{thm}
\begin{proof}
The period length of the continued fraction of $\sqrt{d_i^+}$ is $ 2$ by  theorem \ref{t5}. Since $ p_{-2}=0 ,p_{-1}=1,q_{-2}=1,q_{-1}=0,\ \ p_{k}=a_{k} p_{k-1}+p_{k-2},       q_{k}=a_{k}q_{k-1}+q_{k-2},$ therefore the fundamental solution is of the form $p_{1}^{+i} = a_{1}p_{0}+p_{-1}=2ha^k b^l+1,  q_{1}^{+i}=a_{1}q_{0}+q_{-1}=2h$ for $d_1^+$  by using  lemma \ref{l1}. Similarly the equation
 $ x^2-d_2^+ y^2 = 1 $ has the required solution form due to the theorem \ref{t5} and lemma \ref{l1}.

 Now we assume that $(x_{n-1} , y_{n-1}) $ is a solution, that is, $ x_{n-1}^2-d_2^+y^2_{n-1} = 1.$
Then we have that\\
$ \frac{x_{n}^{+2}}{y_{n}^{+2}}=a^kb^l +\frac{1}{h+\frac{1}{2a^kb^l
+\frac{1}{\ddots \\
\\ \ 2a^kb^l+\frac{1}{h}}}} \\
=a^kb^l +\frac{1}{h+\frac{1}{a^kb^l +a^kb^l +
\frac{1}{h +\frac{1}{\ddots\\
\ \ \  2a^kb^l+\frac{1}{h}}}}}
=a^kb^l+\frac{1}{h+\frac{1}{a^kb^l +\frac{x_{n-1}}{y_{n-1}}}}
 =\frac{x_{n}}{y_{n}}=\frac{(ha^kb^l+1) x_{n-1}+hd{y_{n-1}}}{hx_{n-1}+(ha^kb^l+1) y_{n-1}}$
 \begin{equation}\label{e8}
 \frac{x_{n}^{+2}}{y_{n}^{+2}}=\frac{x_1^{+2} x_{n-1}+y_1^{+2}d{y_{n-1}}}{y_1^{+2}x_{n-1}+x_1^{+2} y_{n-1}}
 \end{equation}

Similarly for $d_1^+$ and $d_i^-$ we can get all positive solutions of the required form.

\end{proof}
 \begin{cor}
     The equation $ x^2-d_i^\pm y^2=-1$ has no positive integer solutions, except $ x^2-d_1^+ y^2=-1$ has solution $(a^{k}b^{l},1)$ only if $c=1.$
       \end{cor}
\begin{proof}
The continued fraction of $\sqrt{d_i^\pm}$ have even length, therefore the equation $ x^2-d_i^\pm y^2 = -1 $ has no solution by lemma \ref{l1}, but if $c=1$, then $\sqrt{d_1^+}$ have odd length, therefore the equation has solution by lemma \ref{l1} and get required solution.
\end{proof}

\begin{thm}\label{s1}
The $n^{th} $ integer solution $ (x_{n}^{\pm i}, y_{n}^{\pm i}) $ of $x^2-d_i^\pm y^2 = -1$  can be given as a linear combination of
$x_i^{\pm i},y_i^{\pm i}$ and $d_i^\pm$ namely, for $ n\geq 2 $
\begin{eqnarray*}
  x_{n}^{\pm i} &=& x_1^{\pm i} x_{n-1}+y_1^{\pm i}d_i^\pm y_{n-1} \\
  y_{n}^{\pm i}&=& y_1^{\pm i}x_{n-1}+x_1^{\pm i}y_{n-1}
\end{eqnarray*}
and also satisfy the recurrence relation for $ n\geq 4 $
\begin{eqnarray*}
  x_{n}^{\pm i} &=&  (2x_1^{\pm i}-1)(x_{n-1} + x_{n-2})-x_{n-3}\\
  y_{n}^{\pm i} &=& (2x_1^{\pm i}-1)(y_{n-1} + y_{n-2})-y_{n-3}
\end{eqnarray*}
\end{thm}
\begin{proof}
 The first assertion is easily seen from \ref{e8}. The second assertion can
be proved by induction on $n$.
\end{proof}

\begin{thm}\label{t7}
All positive integer solutions of the equation $ x^2-d_i^\pm y^2=1 $  are given by
                           $$ (x_n^{\pm i},y_n^{\pm i})=\big(\frac{1}{2}L_{n}(x_1^{\pm i},-1),y_1^{\pm i} f_{n}(x_1^{\pm i},-1)\big) \ \ \ \ \ n=1,2,3,\ldots$$

\end{thm}
 \begin{proof}
Consider the Pell equation  $ x^2-d_i^\pm y^2=1,$ then by theorem \ref{t6} and theorem \ref{t1} all positive solution of the equation are given by  $ x_{n}^{\pm i} + y_{n}^{\pm i}\sqrt{d_i^\pm}=(x_1^{\pm i}+y_1^{\pm i}\sqrt{d_i^{\pm }})^n$

Let $ \alpha_i^{\pm}=x_1^{\pm i}+y_1^{\pm i}\sqrt{d_i^{\pm }} $ and $\beta_i^{\pm}=x_1^{\pm i}-y_1^{\pm i}\sqrt{d_i^{\pm }}$ . Then $ \alpha_i^{\pm} +\beta_i^{\pm} =2x_1^{\pm i} ,    \alpha_i^{\pm}-\beta_i^{\pm}=2y_1^{\pm i}\sqrt{d_i^\pm}$ and $\alpha_i^{\pm}\beta_i^{\pm} = 1 $ .  Therefore   $ x_{n}^{\pm i}+y_{n}^{\pm i}\sqrt{d_i^\pm} = (\alpha_i^{\pm})^n,                  x_{n}^{\pm i}-y_{n}^{\pm i} \sqrt{d_i^\pm}=(\beta_i^{\pm})^n $.\\
Thus it follows that 			    $x_{n}^{\pm i}=\frac{(\alpha_i^{\pm})^n + (\beta_i^{\pm})^n}{2} =  \frac{1}{2}L_{n}(x_1^{\pm i},-1)$ and
$ y_{n}^{\pm i} = \frac{(\alpha_i^{\pm})^n-(\beta_i^{\pm})^n}{2\sqrt{d_i^\pm}} = \frac{(\alpha_i^{\pm})^n-(\beta_i^{\pm})^n}{2y_1^{\pm i}\sqrt{d_i^\pm} } = y_1^{\pm i} f_{n}  (x_1^{\pm i},-1).$\\
  \end{proof}

\begin{thm}\label{t11}
The fundamental solution of the Pell equation  $ x^2-d_i^\pm y^2=4$ is $(x_{1}^{\pm i}, y_{1}^{\pm i})=(2x_1^{\pm i},2y_1^{\pm i}).$
 \end{thm}
    \begin{proof}
 We know that $ a^k b^l =c^m h$ so $(a^kb^l )^2 = (c^m h)^2 \Rightarrow d_i^\pm =c^{2m}h^2\pm i c^m.$
               we will give proof for only $d_2^+=a^{2k}b^{2l}+2c^m.$ If $ c=2s $ is even,  then $ d_2^+\equiv 0 (\mbox{mod 4}) $ and
                 $ \frac{d_2^+}{4}=2^{2m-2}s^{2m}h^2 +2^{m-1}s^m.$ Hence by theorem \ref{t5} and \ref{t6}, it follows that the equation $x^2-(2^{2m-2}s^{2m}h^2+2^{m-1}s^m) y^2=1 $ has the fundamental solution $(2^{m}h^2s^m+1,2h).$ Then, by theorem \ref{t9}, the fundamental solution to the equation $x^2-d_2^+ y^2=4 $ is of the form $(2^{m+1}h^2s^m +2,2h\sqrt{d_2^+})$. Since $c=2s$ and $2^mhs^m =a^kb^l,$ so the the equation has the fundamental solution $(2ha^kb^l+2+2h\sqrt{d_2^+})$.

Assume that $c$ is odd. Then $ d_2^+\equiv 2,3 (\mbox{mod 4}) $ if $a^{2k}b^{2l}$ is even or odd respectively. Thus, by theorem \ref{t7} and \ref{t6}, it follows that the fundamental solution of the equation $ x^2-(2^{2m-2}s^{2m}h^2+2^{m-1}s^m) y^2=1 $ is of the form $(2^{m+1}h^2s^m +2,2h).$ Then, by theorem \ref{t6}, the fundamental solution to the equation $x^2-d_2^+ y^2=4$ is of the form $(2ha^kb^l+2+2h\sqrt{d_2^+})$.
Similarly, we can proof for other values of $d_i^\pm$.
 \end{proof}

\begin{thm}\label{t12}
The equation $  x^2-d_i^\pm y^2=-4 $ has no positive integer solution, except $ x^2-d_1^+ y^2=-4$ when $c=1$
\end{thm}

\begin{proof}
  Let $d_1^+=a^{2k}b^{2l}+c^m$ and assume that $c$ is odd. Then $ d_1^+\equiv 2,3 (\mbox{mod 4}) $ if $a^{2k}b^{2l}$ is even or odd respectively. Thus, by theorem \ref{t8} and theorem \ref{t6} the equation $ x^2-d_1^+ y^2=-4 $ has no solution.

Now suppose that $ c $ is even and the positive integer $f$ and $g$ are the solution of the above equation, then $f^2-d_1^+g^2=-4 $. But $d_1^+$ is even and therefore $f$ and $g$ are even. Since $(a^k b^l)^2=(c^m h)^2$ therefore, $ f^2-(2^{2m-2}s^{2m}h^2+2^{m-1}s^m)g^2=-4 $ and implies that
$ (\frac{f}{2})^2-(2^{2m-2}s^{2m}h^2+2^{m-1}s^m)g^2=-1 $ this is impossible by theorem \ref{t6}.
similarly, for the other equations.
\end{proof}
 \begin{cor}
    The equation $ x^2-d_1^+ y^2=-4$ has positive integer solutions  $(2a^{k}b^{l},2)$ if $c=1$.
       \end{cor}

   \begin{thm}\label{t13}
    All the positive solutions of the equation $ x^2-d_i^\pm y^2=4 $ are given as
$$ (x^{\pm i},y^{\pm i})=\big(L_{n} (2x_1^{\pm i},-1),y_1^{\pm i}f_{n} (2x_1^{\pm i},-1)\big) \ \ \ \  n=1,2,3,\ldots$$
    \end{thm}
\begin{proof}
We know by theorem \ref{t11} that $2x_1^{\pm i}+2y_1^{\pm i}\sqrt{d_i^{\pm }}$ is the fundamental solution of the equation $ x^2-d_i^\pm y^2=4$. Therefore by theorem \ref{t3}, all positive integer solution of the equation $ x^2 - d_i^\pm y^2 = 4 $ are given by $ x_{n}^{\pm i} + y_{n}^{\pm i}\sqrt{{d_i^\pm}}=\frac{1}{2^{n-1}}{\big(2x_1^{\pm i}+2y_1^{\pm i}\sqrt{d_i^{\pm }} \big)^n}=2(\frac{2x_1^{\pm i}+2y_1^{\pm i}\sqrt{d_i^{\pm }}}{2})^n.$ Now, let us consider $ \alpha_i^{\pm}= \frac{2x_1^{\pm i}+2y_1^{\pm i}\sqrt{d_i^{\pm }}}{2} $ and $ \beta_i^{\pm}=\frac{2x_1^{\pm i}-2y_1^{\pm i}\sqrt{d_i^{\pm }}}{2}.$ Then $ \alpha_i^{\pm} + \beta_i^{\pm} =x_1^{\pm i}, \alpha_i^\pm-\beta_i^\pm=2y_1^{\pm i}\sqrt{d_i^{\pm }}$ and $ \alpha_i^\pm\beta_i^\pm=1.$ Thus it is easily seen that $ x_{n} ^{\pm i}+ y_{n}^{\pm i}\sqrt{{d_i^\pm}}=2(\alpha_i^\pm)^n $
     and     $ x_{n}^{\pm i}-y_{n}^{\pm i}\sqrt{{d_i^\pm}}=2(\beta_i^\pm)^n $.
     Therefore
     $ x_{n}^{\pm i}=(\alpha_i^{\pm})^n + (\beta_i^{\pm})^n = L_{n} (2x_1^{\pm i},-1)$ and
     $ y_{n}^{\pm i}=\frac{(\alpha_i^{\pm})^n-(\beta_i^{\pm})^n}{\sqrt{d_i^\pm}}= y_1^{\pm i} \frac{(\alpha_i^\pm)^n-(\beta_i^\pm)^n}{2h \sqrt{d_i^\pm}}=y_1^{\pm i}f_{n} (2x_1^{\pm i},-1). $
   \end{proof}

   Thus we can give the following corollaries.

    \begin{cor}
     If $d_1^+=a^{2k}+a^m,$  then
			$ \sqrt{d_1^+}=[a^k,\overline{2a^{k-m},2a^k}] $
    and if $d_2^+=a^{2k}+2a^m,$  then
			$ \sqrt{d_2^+}=[a^k,\overline{a^{k-m},2a^k}] $
\end{cor}
    \begin{cor}
         If $ d=a^{2k}+2a^m $, then the fundamental solution to the equation $ x^2-d_2^+y^2=1 $ is
     $ x_{1}+y_{1}\sqrt{d_2^+}=a^{2k-m}+a^{k-m}\sqrt{d} $
     and the equation $ x^2-d_2^+y^2=-1 $ has no solutions.
     \end{cor}

   \begin{cor}
     Let $ d_1^+=a^{2k}+a^m $ . Then the fundamental solution to the equation $ x^2-d_1^+y^2=1 $ is
     $ x_{1}+y_{1}\sqrt{d_1^+}=2a^{2k-m}+1+2a^{k-m}\sqrt{d} $
     and the equation $ x^2-d_1^+y^2=-1 $ has no solutions
     \end{cor}
\begin{rk}
\begin{itemize}
  \item If $c=b$ and $k=l=m=1$, then the main results of \cite{G} become the corollaries of our main results.
  \item If $c=a$ and $k=l=m=1$ and $b=1$, then the main results of \cite{P} become the corollaries of our main results.
\end{itemize}
\end{rk}
\section{Main Results 2}
Let us consider the matrix $ M_{\pm i}$ associated with $d_i^\pm$ and corresponding fundamental solution $x_i^\pm \ ,y_i^\pm $ as
\begin{equation}\label{e9}
 M_{\pm i} =\left(
            \begin{array}{cc}
             x_1^{\pm i} & y_1^{\pm i}d_i^\pm \\
              y_1^{\pm i} & x_1^{\pm i}\\
           \end{array}
      \right)
\end{equation}

In the following theorem, we able to determine the $ n^{th} $  power of $ M_{\pm i}$ which we use
it later. (Here, we note that ${n\choose  2j} = {n-2\choose 2j} +{n-2\choose  2j-2}+2{n-2\choose 2j-1}$
for $ j = 1, 2,\ldots ,\frac{n - 2}{2}).$
\begin{thm}\label{t2}
If $n\geq 0$, then the $ n^{th} $  power of $ M_{\pm i}$ is given by
$$ M_{\pm i}^n =\left(
            \begin{array}{cc}
              M_{11}^n &  M_{12}^n \\
              M_{21}^n &  M_{22}^n\\
            \end{array}
          \right) \ \ \ \ \ \ \ \mbox{where}$$
  \begin{description}
  \item[a] If $n$ is even
  \begin{eqnarray*}
   M^n_{11} &=& \sum\limits^{\frac{n}{2}}_{j=0}{n\choose 2j}(x_1^{\pm i})^{n-2j}(y_1^{\pm i})^{2j}(d_i^\pm)^j = M^n_{22} \\
   M^n_{12} &=& \sum\limits^{\frac{n-2}{2}}_{j=0}{n\choose 2j+1}(x_1^{\pm i})^{n-1-2j}(y_1^{\pm i})^{2j+1} (d_i^\pm)^{j+1} \\
   M^n_{21} &=& \sum\limits^{\frac{n-2}{2}}_{j=0}{n\choose 2j+1}(x_1^{\pm i})^{n-1-2j}(y_1^{\pm i})^{2j+1} (d_i^\pm)^j
   \end{eqnarray*}
   \item[b] If $n$ is odd
   \begin{eqnarray*}
   M^n_{11} &=& \sum\limits^{\frac{n-1}{2}}_{j=0}{n\choose 2j}(x_1^{\pm i})^{n-2j}(y_1^{\pm i})^{2j}(d_i^\pm)^j = M^n_{22} \\
   M^n_{12} &=& \sum\limits^{\frac{n-1}{2}}_{j=0}{n\choose 2j+1}(x_1^{\pm i})^{n-1-2j}(y_1^{\pm i})^{2i+1} (d_i^\pm)^{j+1}\\
   M^n_{21} &=& \sum\limits^{\frac{n-}{2}}_{j=0}{n\choose 2j+1}(x_1^{\pm i})^{n-1-2j}(y_1^{\pm i})^{2i+1}(d_i^\pm)^j
   \end{eqnarray*}
   \end{description}
\end{thm}

\begin{proof}(a):
Here we will give the proof for $d_2^+$ by mathematical induction on $n$. \\If $n = 2$, then
$\\ M_{\pm i}^2 =\left(
            \begin{array}{cc}
       h^2a^{2k}b^{2l}+1+2ha^kb^l+h^2d_2^+ & 2hd_2^+(ha^kb^l+1)      \\
                        2hd_2^+(ha^kb^l+1) &  h^2a^{2k}b^{2l}+1+2ha^kb^l+h^2d_2^+  \\
            \end{array}
          \right)
 ,  \ \ where\\$
$ M^2_{11} = \sum\limits^{1}_{j=0}{2\choose 2j}(ha^kb^l+1)^{2-2j}h^{2j}{d_2^+}^j = (ha^kb^l+1)+h^2d_2^+ = M^2_{22}\\
M^2_{12} = \sum\limits^{0}_{j=0}{2\choose 2j+1}(ha^kb^l+1)^{1-2j}h^{2j+1}{ d_2^+}^{j+1}=2hd(ha^kb^l+1)\\
M^2_{21} = \sum\limits^{0}_{j=0}{2\choose 2j+1}(ha^kb^l+1)^{1-2j}h^{2j+1} {d_2^+}^j=2h(ha^kb^l+1)\\$
So it is true for $n = 2$. Let us assume that it is true for $n - 2$, that is,
$M_{\pm i}^{n-2} =\left(
            \begin{array}{cc}
              M^{n-2}_{11} & M^{n-2}_{12} \\
              M^{n-2}_{21} & M^{n-2}_{22} \\
            \end{array}
          \right)$,
where
$ M^{n-2}_{11} = \sum\limits^{\frac{n-2}{2}}_{j=0}{n-2\choose 2j}(ha^kb^l+1)^{n-2-2j}h^{2j}(d_2^+)^j= M^{n-2}_{22}\\
M^{n-2}_{12} = \sum\limits^{\frac{n-4}{2}}_{j=0}{n-2\choose 2j+1}(ha^kb^l+1)^{n-3-2j}h^{2j+1}(d_2^+)^{j+1},
M^{n-2}_{21} = \sum\limits^{\frac{n-4}{2}}_{j=0}{n-2\choose 2j+1}(ha^kb^l+1)^{n-3-2j}h^{2j+1}(d_2^+)^j\\$
We will prove it for $n$, since $M_{\pm i}^n= M_{\pm i}^{n-2}M_{\pm i}^2,$  we get\\
 $  M^{n-2}_{11} [(ha^kb^l+1) + h^2(d_2^+)] +M^{n-2}_{12}[2h(ha^kb^l+1)] =[ (ha^kb^l+1)^{n-2}+{n-2\choose 2}(ha^kb^l+1)^{n-4}h^2(d_2^+) + \cdots +{n-2\choose n-4}(ha^kb^l+1)^2h^{n-4}(d_2^+)^\frac{n-4}{2} + \cdots+h^{n-2}(d_2^+)^\frac{n-2}{2}\cdots][(ha^kb^l+1) + h^2(d_2^+)]+
[{n-2\choose 1}(ha^kb^l+1)^{n-3}h(d_2^+) + \cdots +{n-2\choose 3}(ha^kb^l+1)^{n-5}h^3(d_2^+)^2 + \cdots +{n-2\choose n-5}(ha^kb^l+1)^3h^{n-5}(d_2^+)^\frac{n-4}{2}+ \cdots +{n-2\choose n-3}(ha^kb^l+1)h^{n-3}(d_2^+)^\frac{n-2}{2} ] [2h(ha^kb^l+1)]\\
   =[(ha^kb^l+1)^{n-2}+{n-2\choose 2}(ha^kb^l+1)^{n-4}h^2(d_2^+) + \cdots +{n-2\choose n-4}(ha^kb^l+1)^2h^{n-4}(d_2^+)^\frac{n-4}{2} + \cdots+h^{n-2}(d_2^+)^\frac{n-2}{2}\cdots][(ha^kb^l+1) + h^2(d_2^+)]+
[{n-2\choose 1}(ha^kb^l+1)^{n-3}h(d_2^+) + \cdots +{n-2\choose 3}(ha^kb^l+1)^{n-5}h^3(d_2^+)^2 + \cdots +{n-2\choose n-5}(ha^kb^l+1)^3h^{n-5}(d_2^+)^\frac{n-4}{2}+ \cdots +{n-2\choose n-3}(ha^kb^l+1)h^{n-3}(d_2^+)^\frac{n-2}{2} ] [2h(ha^kb^l+1)]\\
   = (2ha^kb^l-1)^n + {n - 2\choose  2} + 1 + 2{n - 2\choose  1} ](2ha^kb^l-1)^{n-2}4h^2(d_2^+)
+{n - 2\choose 4} + {n - 2\choose 2} + 2{n - 2\choose 3}(ha^kb^l+1)^{n-4}h^4(d_2^+)^2+\cdots+
(ha^kb^l+1)^4h^{n-4}(d_2^+)^{\frac{n-4}{2}}\\+[1 + {n - 2\choose  n - 4} + 2{n - 2\choose  n - 3} (ha^kb^l+1)^2h^{n-2}(d_2^+)^{\frac{n-2}{2}}
+h^n(d_2^+){\frac{n}{2}}\\
   = (ha^kb^l+1)^n + {n\choose 2}(ha^kb^l+1)^{n-2}h^2(d_2^+) + {n\choose 4}(ha^kb^l+1)^{n-4}h^4(d_2^+)^2
+\cdots+ {n\choose n - 2}(ha^kb^l+1)^2h^{n-2}(d_2^+)^{\frac{n-2}{2}}+ (h)^n(d_2^+)^{\frac{n}{2}}
=\sum\limits^{\frac{n}{2}}_{j=0}{n\choose  2j}(ha^kb^l+1)^{n-2j}h^{2j}(d_2^+)^j=M^n_{12}$

 Similarly, it can be shown that for $M_{12}^n,M_{21}^n$, thus
$M_{\pm i}^n =\left(
            \begin{array}{cc}
              M^n_{11} & M^n_{12} \\
              M^n_{21} & M^n_{22}
            \end{array}
          \right)$
as claimed and the other cases for $d_i^\pm$ can be proved similarly.
\end{proof}

In the following theorem, we will show that the $ n^{th} $ integer solution
$ (x_{n}^{ \pm i}, y_{n}^{\pm i}) $ of $ F_{\Delta_{d_i^\pm}} (x, y) = 1 $ can be deduce via $d_i^\pm.$
\begin{thm}
 The $ n^{th} $ integer solution of $ F\Delta_{d_i^\pm} (x, y) = 1 $ is $(x_{n}^{\pm i}, y_{n}^{\pm i}),$ where\\
$x_{n}^{\pm i} = \left\{
          \begin{array}{ll}
           \sum\limits^{\frac{n}{2}}_{j=0} {n\choose2j}(x_1^{\pm i})^{n-2j}(y_1^{\pm i})^{2j}(d_i^\pm)^j, & \hbox{if n is even;} \\
    \sum\limits^{\frac{n-1}{2}}_{j=0}{n\choose 2j}(x_1^{\pm i})^{n-2j}(y_1^{\pm i})^{2j}(d_i^\pm)^j, & \hbox{if n is odd.} \\
          \end{array}
        \right.\\$
$y_{n}^{\pm i} = \left\{
          \begin{array}{ll}
           \sum\limits^{\frac{n-2}{2}}_{j=0} {n\choose 2j+1}(x_1^{\pm i})^{n-1-2j}(y_1^{\pm i})^{2j+1}(d_i^\pm)^j, & \hbox{if n is even;} \\
    \sum\limits^{\frac{n-2}{2}}_{j=0}{n\choose 2j}(x_1^{\pm i})^{n-1-2j}(y_1^{\pm i})^{2j+1}(d_i^\pm)^j, & \hbox{if n is odd.} \\
          \end{array}
        \right.$
  \end{thm}

\begin{proof}
 From \ref{s1}, we can write the $n^{th}$ solution for $d_2^+=a^{2k}b^{2l}+2c^m$\\
  $\left(
                              \begin{array}{c}
                                x_n^{+2} \\
                                y_n^{+2} \\
                              \end{array}
                            \right)
                            = \left(
   \begin{array}{cc}
     ha^kb^l+1 & hd_1 \\
     h & ha^kb^l+1 \\
   \end{array}
 \right)^n
 \left(
            \begin{array}{c}
             ha^kb^l \\
            h \\
            \end{array}
          \right)=\left(
                    \begin{array}{c}
                      M_{11}^n+M_{12}^n \\
                       M_{21}^n+M_{22}^n \\
                    \end{array}
                  \right) $
after simplification, we get the required solution.
\end{proof}
Now we can consider the Pell form $ F_{\Delta_{d_i^\pm}}$  and note that this form is not reduced
since $ |\sqrt{4d} -1| > 0.$ So we can give the following
theorem related to reduction of $ F_{\Delta_{d_i^\pm}}:$
\begin{thm}
\begin{description}
\item[i] The reduction of $ F_{\Delta_{d_1^+}}$  is \\$\rho^2(F_{\Delta_{d_1^+}}) =  (1,2hc^m,-c^m) . $
\item[ii] The reduction of $ F_{\Delta_{d_2^+}}$  is \\$\rho^2(F_{\Delta_{d_2^+}}) =  (1,2hc^m,-2c^m) . $
\item[iii] The reduction of $ F_{\Delta_{d_2^-}}$  is \\$\rho^2(F_{\Delta_{d_1^-}}) = (1,2hc^m-2,1-(2h-1)c^m) . $
\item[iv] The reduction of $ F_{\Delta_{d_1^-}}$  is \\$\rho^2(F_{\Delta_{d_2^-}}) =  (1,2hc^m-2,1-2(h-1)c^m) . $
\end{description}
\end{thm}

\begin{proof}
Let $ F_{\Delta_{d_1^+}} = F_{\Delta_{d_{1,0}^+}} = (1, 0,-d_1^+).$ Then from (\ref{e5}), we get $ r_{0} = 0$
and hence from (\ref{e4}), we have $\rho^1(F_{\Delta_{d_1^+}})= (-d_1^+, 0, 1)$
which is not reduced. If we apply the reduction algorithm to $\rho^2(F_{\Delta_{d_1^+}} )$ again, then
we find that $ r_{1} =a^kb^l$ and so $\rho^2(F_{\Delta_{d_1^+}} ) = (1,2hc^m,-c^m)$
which is reduced. Similarly, we can get the reduction for others forms.
\end{proof}
Now we can consider the cycle and proper cycle of $\rho^2(F_{\Delta_{d_i^\pm}}).$

\begin{thm}
Let us consider the reduction $\rho^2(F_{\Delta_{d_i^+}})$ of $(F_{\Delta_{d_i^+}}).$ Then
\begin{enumerate}
\item The cycle of $\rho^2(F_{\Delta_{d_1^+}} )$ is $ (1,2hc^m,-c^m) \sim (c^m,2hc^m,-1).$
\item The cycle of $\rho^2(F_{\Delta_{d_1^-}} )$ is $ (1,2hc^m-2,1-(2h-1)c^m) \sim (2hc^m-c^m-1,(2h-2)c^m,-c^m)\sim (c^m,2hc^m-2c^m,-2hc^m+c^m+1)\sim (2hc^m-c^m-1,2hc^m-2,-1).$
\item The cycle of $\rho^2(F_{\Delta_{d_2^+} })$ is $ (1,2hc^m-2,1-(2h-1)c^m) \sim (2c^m,2hc^m,-1).$
\item The cycle of $\rho^2(F_{\Delta_{d_2^-} })$ is $ (1,2hc^m-4,1-2(h-1)c^m) \sim (2hc^m-2c^m-1,2(h-2)c^m,-2c^m)\sim (2c^m,2hc^m-4c^m,-2hc^m+2c^m+1)\sim (2hc^m-2c^m-1,2hc^m-4,-1).$
\end{enumerate}
\end{thm}

\begin{proof}
 Let $\rho^2(F_{\Delta_{d_1^+}})  = \rho^2(F_{\Delta_{d_{1,0}^+}})=(1,2a^kb^l,-a^kb^l).$
Then from (\ref{e6}), we get $ s_{0} = 2h $ and thus
$\rho^2(F_{\Delta_{d_{1,1}^+}})=(a^kb^l,2a^kb^l,-1)$. Again from from (\ref{e6}), we get $s_{1} = 2a^kb^l  $ and hence
$\rho^2(F_{\Delta_{d_{1,2}^+}})=(1,2a^kb^l,-a^kb^l)= \rho^2(F_{\Delta_{d_1^+,0}}).$
So the cycle of $ \rho^2(F_{\Delta_{d_1^+}}) $ is $\rho^2(F_{\Delta_{d_{1,1}^+}})\sim \rho^2(F_{\Delta_{d_{1,1}^+}}).$ Note that $ l = 2, $ therefore from
theorem \ref{t1} the proper cycle of $ \rho^2(F_{\Delta_{d_1^+}})$ is
$ (1,2a^kb^l,-a^kb^l)) \sim  (-2a^kb^l,2a^kb^l,1)$ of length $2$. Similarly for  $\rho^2(F_{\Delta_{d_2^+}}) $ and$\rho^2(F_{\Delta_{d_i^-}})$, we can obtained the required cycle.
\end{proof}
\begin{cor}
The proper cycle of  $\rho^2(F_{\Delta_{d_1^+}}) $ is $(1,2hc^m,-c^m) \sim (-c^m,2hc^m,1)$
of length $2$.\\
The proper cycle of   $\rho^2(F_{\Delta_{d_2^+}}) $ is $ (1,2hc^m,-2c^m) \sim (2c^m,2hc^m,-1).$
of length $2$.\\
The proper cycle of \\ $\rho^2(F_{\Delta_{d_1^-}}) $ is$\rho^2(F_{\Delta_{d_1^-}} )$ is $ (1,2hc^m-2,1-(2h-1)c^m) \sim (-2hc^m+c^m+1,(2h-1)c^m,c^m)\sim (-c^m,2hc^m-2c^m,2hc^m-c^m-1)\sim (-2hc^m+c^m+1,2hc^m-2,1).$
of length 4.\\
The proper cycle of \\$\rho^2(F_{\Delta_{d_2^-}}) $ is  $ (1,2hc^m-2,1-2(h-1)c^m) \sim (-2hc^m+2c^m+1,2(h-2)c^m,2c^m)\sim (-2c^m,2hc^m-4c^m,2hc^m-2c^m-1)\sim (-2hc^m+2c^m+1,2hc^m-4,1).$of length 4.\\
\end{cor}

Now we consider the proper automorphisms of $ F_{\Delta_{d_i^\pm}} .$  To get this we first
consider the following representations of the action of the group $\Gamma$

              $ gF_{\Delta_{d_i^\pm}}=\left(
            \begin{array}{cc}
              x_1^{\pm i} & y_{1}^{\pm i} \\
              y_1^{\pm i}d_i^\pm & x_1^{\pm i} \\
            \end{array}
          \right)$

Then we can give the following theorem which can be proved as in the same way
that theorem \ref{t2} was proved.
\begin{thm}
Let  $d_i^\pm$  denote non-zero square free positive integer . Then
\begin{description}
\item[i] The set of proper automorphisms of $ F_{\Delta_{d_i^\pm}} $  is\\
$ Aut^+( F_{\Delta_{d_i^\pm}}  ) = \{\pm g^n_{F_{\Delta_{d_i^\pm}}}: n\in Z\}.$
\item[ii] The integer solutions of $ F_{\Delta_{d_i^\pm}} (x, y) = 1 $ are $ (x_{n}^{\pm i}, y_{n}^{\pm i}),$ where \\
$ \left(
  \begin{array}{c}
    x_{n}^{\pm i} \\
    y_{n}^{\pm i} \\
  \end{array}
\right)$
$= (g^n_{F_{\Delta_{d_i^\pm}}})\left(
  \begin{array}{c}
    1\\
    0 \\
  \end{array}
\right)$
 for $ n \geq 1.$
\end{description}
\end{thm}
\begin{rk}
If $c=a$ and $k=l=m=1$ and $b=1$, then the main results of \cite{Atz} become the corollaries of our main results.
\end{rk}

\end{document}